\documentclass[11pt]{article}

\usepackage[utf8]{inputenc}
\usepackage[T1]{fontenc}
\usepackage{lmodern}
\usepackage{microtype}

\usepackage[margin=1.15in]{geometry}

\usepackage{amsmath,amssymb,amsthm,mathtools}

\usepackage[colorlinks=true,
            linkcolor=blue,
            citecolor=blue,
            urlcolor=blue]{hyperref}
\usepackage{booktabs}
\usepackage{enumitem}

\theoremstyle{plain}
\newtheorem{theorem}{Theorem}[section]
\newtheorem{prop}[theorem]{Proposition}
\newtheorem{lemma}[theorem]{Lemma}
\newtheorem{cor}[theorem]{Corollary}

\theoremstyle{definition}

\theoremstyle{remark}
\newtheorem{remark}[theorem]{Remark}

\newcommand{\e}{\mathrm{e}}

\DeclareMathOperator{\argmax}{arg\,max}

\title{Asymptotic Rounding Laws for Optimal Thresholds in Secretary Problems}

\author{Ra\'ul S\'anchez Gal\'an}

\date{\today}

\begin{document}

\maketitle

\begin{abstract}
We study second-order asymptotics for optimal thresholds in secretary-type
problems. Under a unimodality assumption and a local affine expansion of the
discrete payoff increments, we establish a general principle showing that
sufficiently accurate rational approximations to the leading threshold
proportion yield exact optimal integer thresholds. Continued-fraction
convergents provide an immediate application. We also derive an asymptotic
rounding law and apply the results to several variants of the secretary problem with fixed and random horizons.
\end{abstract}

% ============================================================
% ============================================================
\section{Introduction}
% ============================================================

The classical secretary problem is a standard example in optimal stopping
theory; see, for example,
\cite{GilbertMosteller1966} and the historical account
\cite{Ferguson1989}. A fixed number \(n\) of applicants is observed in
uniformly random order, and the objective is to maximize the probability of
selecting the best one. The optimal rule is of threshold type: reject the
first \(r\) applicants and then accept the first subsequent record. Its
optimal threshold satisfies
\[
\frac{r_n^{\mathrm{opt}}}{n}\longrightarrow\frac1{\e}.
\]
Many variants exhibit the same phenomenon, with an optimal threshold
asymptotic to \(\ell n\) for some constant \(\ell\in(0,1)\).

The first-order constant \(\ell\), however, does not determine how the
continuous approximation is converted into an exact integer threshold. The
purpose of this paper is to study this second-order question. If \(V_n(r)\)
denotes the expected payoff of a threshold rule, we consider its discrete
increments
\[
\Delta_n(r)=V_n(r)-V_n(r-1)
\]
and a positive rescaling \(D_n(r)\). Our main assumption is a local
asymptotic expansion of the form
\[
D_n(\ell n+u)
=
A(u-c)+o(1),
\qquad
A<0,\qquad 0<c<1,
\]
uniformly for \(|u|<2\). The constant \(c\) determines the position of the
sign change of the increments relative to \(\ell n\).

Under a natural unimodality hypothesis, this local expansion yields an
asymptotic rounding law. For all sufficiently large \(n\), the optimal
threshold lies in
\[
\{\lfloor\ell n\rfloor,\lfloor\ell n\rfloor+1\},
\]
and, away from a shrinking transition region,
\[
r_n^{\mathrm{opt}}
=
\lfloor\ell n+c\rfloor.
\]
Thus \(c\) is the second-order parameter governing the rounding of the
leading approximation \(\ell n\). Moreover, if integers \(p_j,q_j\) satisfy
\[
q_j\longrightarrow\infty,
\qquad
p_j-\ell q_j\longrightarrow0,
\]
then \(p_j\) is eventually the unique optimal threshold for the problem with
parameter \(q_j\). Continued-fraction convergents provide an immediate
arithmetic application, generalizing the exact \(1/\e\) phenomenon
established for the classical secretary problem in~\cite{Sanchez2027}.

We apply the general result to several secretary-type models. These include
fixed-horizon problems with harmonic threshold increments, the secretary
problem with uncertain employment, power-biased random horizons, uniform
random-horizon win--lose--or--draw problems, and the uniform random-horizon
Best-or-Worst and Postdoc problems. In each case, we identify the leading
proportion \(\ell\) and the second-order constant \(c\), and verify the
required local expansion. 
The Poisson model illustrates the shifted form of the general principle.
Although its leading threshold proportion is again \(\ell=1/\e\), its
second-order constant lies in \((-1,0)\). Consequently, if \(p/q\) is a
sufficiently large continued-fraction convergent of \(1/\e\), then \(p-1\),
rather than \(p\), is the unique optimal threshold for the Poisson model
with mean \(q\).

% ============================================================
\section{General continued-fraction threshold principle}
\label{sec:general_principle}
% ============================================================

For convenience, we introduce the notation,
\begin{equation*}
    I_n := \{0,1,\dots,n\}, \qquad I_n^* := \{1,\dots,n\}.
\end{equation*}

A \emph{single-threshold problem with parameter $n$} is an optimization problem in which the stopping rules are given by an integer threshold $r\in I_n$.

For each \(n\) and \(r\in I_n\), let \(X_{n,r}\) denote the random payoff
obtained by using threshold \(r\), and set
\[
V_n(r):=\mathbb E[X_{n,r}],
\]
where the expectation is taken over all sources of randomness in the model.
We define the optimal threshold to be the largest maximizer,
\[
r_n^{\mathrm{opt}}
:=
\max \{\argmax_{r\in I_n}V_n(r)\}.
\]
Throughout the paper, optimality refers to optimality within the class of single-threshold stopping rules.

In the secretary problems considered below, \(r\) is the number of initial
applicants rejected before the acceptance rule is activated. In the classical
secretary problem, \(X_{n,r}\) is the indicator of the event that the best
applicant is selected, so \(V_n(r)\) is the corresponding success
probability. In variants with rewards and penalties, \(X_{n,r}\) is the
payoff prescribed by the model.

A first-order description of the optimal threshold is given, when the limit
exists, by its leading asymptotic proportion
\begin{equation}\label{eq:first_order_asymp}
\ell
:=
\lim_{n\to\infty}
\frac{r_n^{\mathrm{opt}}}{n}
\in(0,1).
\end{equation}
In the applications below, this limit will be identified explicitly for each
model.

We define the discrete increments by
\[
\Delta_n(r)
:=
V_n(r)-V_n(r-1),
\qquad
r\in I_n^*.
\]

The key observation is that 
if the sequence
\(\{V_n(r)\}_{r=0}^n\) is \emph{unimodal}, i.e., there exist
\(0\leq a_n\leq b_n\leq n\) such that
\[
V_n(0)\leq\cdots\leq V_n(a_n)
=\cdots=V_n(b_n)
\geq\cdots\geq V_n(n),
\]
then an optimal threshold is detected by a change of sign of the increments (from positive to negative). If \(p_j/q_j\) approximates the leading proportion \(\ell\), then \(p_j\) will be the exact optimal threshold whenever
\(\Delta_{q_j}(p_j)>0\) and
\(\Delta_{q_j}(p_j+1)<0\). The next result gives a simple local asymptotic
condition guaranteeing precisely this sign pattern.

\begin{theorem}
\label{thm:general_threshold_principle}
Consider a family of single-threshold problems $\{ \mathcal F_n\}_{n \geq 1}$
where each \(\mathcal F_n\) has payoff function \(V_n\). Assume that, for each
\(n\), the finite sequence
$\{V_n(r)\}_{r=0}^{n}$
is unimodal. Suppose that there exists a sequence \(\{k_n\}_{n\geq1}\) of positive constants such that the rescaled increments 
\[
D_n(r)=k_n\Delta_n(r)
\]
satisfy the following uniform local asymptotic expansion: there exist constants \(A<0\), \( c, \ell \in(0,1)\), such that, 
\begin{equation}\label{eq:general_local_expansion}
\varepsilon_n
:=
\sup_{\substack{|u|<2\\ \ell n+u\in I_n^*}}
\left|
D_n(\ell n+u)-A(u-c)
\right|
\longrightarrow0
\qquad\text{as }n\to\infty.
\end{equation}

Let \((p_j,q_j)_{j\geq1}\) be a sequence of integer pairs such that \[ q_j\longrightarrow\infty, \qquad p_j-\ell q_j\longrightarrow0. \] 

Then \(p_j\) is the unique optimal threshold for \(\mathcal F_{q_j}\) for all sufficiently large \(j\).\end{theorem}

\begin{proof}
Set
\[
s_j:=p_j-\ell q_j,
\]
so that \(s_j\to0\). Since \(\ell\in(0,1)\) and \(q_j\to\infty\), we have
\[
p_j,p_j+1\in I_{q_j}^*
\]
for all sufficiently large \(j\). Moreover, both \(s_j\) and \(1+s_j\)
eventually have absolute value less than \(2\). Hence
\eqref{eq:general_local_expansion} gives
\[
D_{q_j}(p_j)
=
A(s_j-c)+o(1)
\longrightarrow
-Ac>0
\]
and
\[
D_{q_j}(p_j+1)
=
A(1+s_j-c)+o(1)
\longrightarrow
A(1-c)<0.
\]
Thus, for all sufficiently large \(j\),
\[
D_{q_j}(p_j)>0,
\qquad
D_{q_j}(p_j+1)<0.
\]

Since \(D_{q_j}=k_{q_j}\Delta_{q_j}\) and \(k_{q_j}>0\), the same sign
relations hold for the original increments. Therefore,
\[
V_{q_j}(p_j-1)
<
V_{q_j}(p_j)
>
V_{q_j}(p_j+1).
\]
By unimodality, \(p_j\) is the unique global maximizer. 
\end{proof}

\begin{cor}
If \(\ell\) is irrational, then every sufficiently large
continued-fraction convergent \(p/q\) of \(\ell\) gives the unique optimal
threshold \(p\) for \(\mathcal F_q\). If \(\ell=a/b\) is rational in lowest
terms, then, for every sufficiently large \(m\), the threshold \(am\) is the
unique optimal threshold for \(\mathcal F_{bm}\).
\end{cor}

Theorem~\ref{thm:general_threshold_principle} also applies to accurate
integer approximations that are not necessarily continued-fraction
convergents. The classical secretary problem provides a simple example.

\begin{cor}
\label{cor:derangement_threshold}
For every sufficiently large integer \(n\), the classical secretary problem
with \(n!\) applicants has the \(n\)-th derangement number \(!n\) as its unique optimal
threshold.
\end{cor}

\begin{proof}
The alternating-series estimate gives
\[
\left|
!n-\frac{n!}{\e}
\right|
=
n!\left|
\sum_{k=n+1}^{\infty}\frac{(-1)^k}{k!}
\right|
<
\frac1{n+1}
\longrightarrow0.
\]
The result follows by applying
Theorem~\ref{thm:general_threshold_principle} with
$(p_n,q_n)=(!n,n!)$ and using the verification of the theorem's hypotheses for the classical secretary problem in Subsection~\ref{subsec:fixed_harmonic}.
\end{proof}

\begin{remark}
\label{remark:Shifted}
The affine form of the limiting profile is not essential for the
exact-threshold conclusion of
Theorem~\ref{thm:general_threshold_principle}. Retain the notation, the
unimodality assumption, and the definitions of \(D_n\) and \(\ell\) above. Suppose that, on an open interval 
\(U\subset\mathbb R\),
\[
D_n(\ell n+u)\longrightarrow g(u)
\]
locally uniformly over admissible integer arguments, where
\(g:U\to\mathbb R\) is continuous. Let \((p_j,q_j)\) be a sequence of
integer pairs such that
\(
q_j \to\infty\) and 
\(p_j-\ell q_j \to s \).
If \(h\in\mathbb Z\) satisfies
\[
s+h,\ s+h+1\in U,
\qquad
g(s+h)>0>g(s+h+1),
\]
then we obtain the \emph{shifted form of the principle}
\begin{equation}\label{eq:shifted_form}
p_j+h\ \text{is the unique optimal threshold for }\mathcal F_{q_j}
\text{ for all sufficiently large }j.
\end{equation}
Indeed, the two offsets
\(p_j+h-\ell q_j\) and
\(p_j+h+1-\ell q_j\) converge to \(s+h\) and \(s+h+1\), respectively. Hence local uniform
convergence gives
\[
D_{q_j}(p_j+h)\longrightarrow g(s+h)>0,
\qquad
D_{q_j}(p_j+h+1)\longrightarrow g(s+h+1)<0.
\]
Positive scaling transfers these signs to the discrete increments, and
unimodality gives \eqref{eq:shifted_form}.

The exact-threshold conclusion of
Theorem~\ref{thm:general_threshold_principle} is recovered by taking
\[
s=h=0,
\qquad
U=(-2,2),
\qquad
g(u)=A(u-c).
\]
More generally, for the affine profile \(g(u)=A(u-c)\), with \(s=0\), the
sign condition is equivalent to
\[
h<c<h+1,
\]
recovering \eqref{eq:shifted_form}. If \(c\in\mathbb Z\), one of the two limiting values at the adjacent
integer offsets vanishes, so a sharper remainder estimate or a higher-order expansion may be needed. Thus a stable sign crossing, rather than affinity, is the essential
structure for exact threshold identification. The affine formulation is
nevertheless retained in
Theorem~\ref{thm:general_threshold_principle} because the local expansions in the applications naturally take this form and because it leads directly to the rounding law below.
\end{remark}

\begin{remark}
For the conclusion of
Theorem~\ref{thm:general_threshold_principle}, the constant \(2\) in
\eqref{eq:general_local_expansion} may be replaced by any fixed \(M>1\).
Indeed, the proof only uses the offsets \(s_j\to0\) and
\(1+s_j\to1\), both of which eventually lie in \((-M,M)\). On the other hand, the rounding theorem below uses the offset
\(2-\{\ell n\}\), which can be arbitrarily close to \(2\) from below.
\end{remark}

\begin{remark}
If an explicit upper bound for \(\varepsilon_n\) is available, the
principle gives a verifiable sufficient condition for a given integer
\(p\) to be the unique optimal threshold. Let
\[
m:=\min\{c,1-c\},
\qquad
d:=\min\{\ell,1-\ell\}.
\]
For integers \(p,q\), set \(s:=p-\ell q\). If
\[
|s|\leq\frac{m}{2},
\qquad
q\geq\frac{2}{d},
\qquad
\varepsilon_q<\frac{|A|m}{2},
\]
then \(p\) is the unique optimal threshold for \(\mathcal F_q\).
Indeed, these conditions ensure that \(p,p+1\in I_q^*\), that the local
expansion applies, and that
\[
D_q(p)\geq\frac{|A|m}{2}-\varepsilon_q>0,
\qquad
D_q(p+1)\leq-\frac{|A|m}{2}+\varepsilon_q<0.
\]
\end{remark}

Here and below,
\[
\{x\}:=x-\lfloor x\rfloor\in[0,1)
\]
denotes the fractional part of \(x\).

The next theorem shows that the optimum is one of the two consecutive
integers \(\lfloor\ell n\rfloor\) and \(\lfloor\ell n\rfloor+1\) and,
away from a shrinking neighborhood of
\(\{\ell n\}=1-c\), is given by
\[
r_n^{\mathrm{opt}}=\lfloor\ell n+c\rfloor.
\]

\begin{theorem}
\label{thm:asymptotic_rounding}
Under the hypotheses of
Theorem~\ref{thm:general_threshold_principle}, for all sufficiently large \(n\),
\[
r_n^{\mathrm{opt}}\in\{\lfloor\ell n\rfloor,\lfloor\ell n\rfloor+1\}.
\]
Moreover,
\[
\{\ell n\}<1-c-\frac{\varepsilon_n}{|A|}
\quad\Longrightarrow\quad
r_n^{\mathrm{opt}}=\lfloor\ell n\rfloor,
\]
whereas
\[
\{\ell n\}>1-c+\frac{\varepsilon_n}{|A|}
\quad\Longrightarrow\quad
r_n^{\mathrm{opt}}=\lfloor\ell n\rfloor+1.
\]
Consequently, whenever
\[
\left|\{\ell n\}-(1-c)\right|>\frac{\varepsilon_n}{|A|},
\]
the optimal threshold is unique and satisfies
\[
r_n^{\mathrm{opt}}
=
\lfloor\ell n+c\rfloor.
\]
\end{theorem}

\begin{proof}
Let
\[
m_n:=\lfloor\ell n\rfloor,
\qquad
\theta_n:=\{\ell n\},
\qquad
\eta_n:=\frac{\varepsilon_n}{|A|}.
\]
Since \(\ell\in(0,1)\), the thresholds
\(m_n,m_n+1,m_n+2\) belong to \(I_n^*\) for all sufficiently
large \(n\). Moreover, since \(\varepsilon_n\to0\), we may assume that
\[
\varepsilon_n
<
|A|\min\{c,1-c\}.
\]
Since
\(
m_n=\ell n-\theta_n
\),
the definition of \(\varepsilon_n\) gives
$\left|
D_n(m_n)-A(-\theta_n-c)
\right|
\leq\varepsilon_n$.
Consequently,
\[
D_n(m_n)
\geq
|A|(\theta_n+c)-\varepsilon_n
\geq
|A|c-\varepsilon_n
>0.
\]
We next show that every optimal threshold belongs to
\(\{m_n,m_n+1\}\).

If \(\theta_n=0\), then
\(
\left|
D_n(m_n+1)-A(1-c)
\right|
\leq\varepsilon_n
\),
and hence
\[
D_n(m_n+1)
\leq
-|A|(1-c)+\varepsilon_n
<0.
\]
Thus, by unimodality, \(m_n\) is the unique optimal threshold in this
case.

Suppose now that \(\theta_n>0\). Since
\(m_n+2=\ell n+(2-\theta_n)\) and \(2-\theta_n<2\), we can apply the local expansion \eqref{eq:general_local_expansion}. Therefore
\(
\left|
D_n(m_n+2)-A(2-\theta_n-c)
\right|
\leq\varepsilon_n
\) and since
\(
2-\theta_n-c>1-c>0
\)
and \(A<0\), it follows that
\[
D_n(m_n+2)
<
-|A|(1-c)+\varepsilon_n
<0.
\]
Together with \(D_n(m_n)>0\), unimodality implies that every optimal
threshold belongs to \(\{m_n,m_n+1\}\).
It remains to compare these two thresholds. Since $m_n+1=\ell n+(1-\theta_n)$,
\[
\left|
D_n(m_n+1)-A(1-\theta_n-c)
\right|
\leq\varepsilon_n.
\]

If $\theta_n<1-c-\eta_n$,
then
$1-\theta_n-c>\eta_n$,
and therefore
\[
D_n(m_n+1)
\leq
-|A|(1-\theta_n-c)+\varepsilon_n
<0.
\]
Thus
\[
V_n(m_n)>V_n(m_n+1),
\]
so \(m_n\) is the unique optimal threshold.

Likewise, if
$\theta_n>1-c+\eta_n$,
then
$1-\theta_n-c<-\eta_n$,
and hence
\[
D_n(m_n+1)
\geq
-|A|(1-\theta_n-c)-\varepsilon_n
>0.
\]
It follows that
\[
V_n(m_n+1)>V_n(m_n),
\]
so \(m_n+1\) is the unique optimal threshold.

Finally,
\[
\lfloor\ell n+c\rfloor
=
\begin{cases}
m_n, & \theta_n<1-c,\\
m_n+1, & \theta_n\geq1-c.
\end{cases}
\]
Therefore,
$r_n^{\mathrm{opt}}
=
\lfloor\ell n+c\rfloor$
whenever
$\left|\{\ell n\}-(1-c)\right|>\eta_n$.
\end{proof}

\begin{remark}[Probabilistic interpretation of \(c\)]
Suppose that \(\ell\) is irrational, and let \(J_N\) be uniformly
distributed on \(\{1,\dots,N\}\). Since \(\varepsilon_n/|A|\to0\), the
equidistribution of \(\{\ell n\}\) implies that the exceptional set
$\left\{
n\geq1:
\left|\{\ell n\}-(1-c)\right|
\leq\frac{\varepsilon_n}{|A|}
\right\}$ has natural density zero. Consequently,
\[
\mathbb P\!\left(
r_{J_N}^{\mathrm{opt}}=\lfloor\ell J_N+c\rfloor
\right)
\longrightarrow1.
\]
Moreover,
\[
\mathbb P\!\left(
r_{J_N}^{\mathrm{opt}}=\lceil\ell J_N\rceil
\right)
\longrightarrow c,
\qquad
\mathbb P\!\left(
r_{J_N}^{\mathrm{opt}}=\lfloor\ell J_N\rfloor
\right)
\longrightarrow1-c.
\]
Thus \(c\) is the limiting probability that the optimal threshold rounds
\(\ell J_N\) upward.
\end{remark}

% ============================================================
\section{Applications of the general principle}
\label{sec:applications}
% ============================================================

We now verify that Theorem~\ref{thm:general_threshold_principle} is applicable to several secretary models. 
Throughout this section we use the notation
\[
S_r^{n-1}:=\sum_{j=r}^{n-1}\frac1j,
\qquad
S_r^{n}:=\sum_{j=r}^{n}\frac1j,
\qquad
T_r^{n}:=\sum_{j=r}^{n}\frac1{j^2},
\]
with empty sums interpreted as zero. 
Euler--Maclaurin summation gives
\begin{equation}\label{eq:EM_expansion}
S_r^{n-1}
=
\log\left(\frac nr\right)
+
\frac1{2r}
-
\frac1{2n}
+
O\left(r^{-2}+n^{-2}\right),
\end{equation}
and
\begin{equation}\label{eq:T_EM_expansion}
T_r^n
=
\frac1r-\frac1n
+
O\left(r^{-2}+n^{-2}\right).
\end{equation}
In particular, if \(r=\ell n+O(1)\), both remainders are \(O(n^{-2})\).

% ------------------------------------------------------------
%%%%%%

% ------------------------------------------------------------
\subsection{Fixed-horizon models with harmonic threshold increments}
\label{subsec:fixed_harmonic}
% ------------------------------------------------------------

The classical secretary problem
\cite{GilbertMosteller1966,Lindley1961}, the secretary problem with interview
cost~\cite{BartoszynskiGovindarajulu1978}, and the Win--Lose--or--Draw problem considered in~\cite{BayonEtAl2023} have payoff increments governed by
the comparison of a truncated harmonic sum with a constant. We describe an affine record-payoff framework that contains all three models.

Suppose that a threshold rule rejects the first \(r\) candidates and then
accepts the first subsequent record, where \(1\leq r\leq n\). Assume that,
conditional on accepting a record at step \(k\), its expected payoff is
\[
a\frac{k}{n}+b,     
\]
where \(a>0\), while making no selection has payoff \(w\). For \(1\leq r<k\leq n\), the rule stops at step \(k\) precisely when the
\(k\)-th candidate is a record and the best candidate among the first \(k-1\) observations lies among the first \(r\). Hence, the probability
that the rule stops at \(k>r\) is
\[
\frac{r}{k(k-1)},
\]
whereas the probability of making no selection is \(r/n\). Consequently, the
expected payoff of threshold \(r\) is
\begin{align}
V_n(r)
&=
\sum_{k=r+1}^{n}
\frac{r}{k(k-1)}
\left(
a\frac{k}{n}+b
\right)
+
w\frac{r}{n}
=
b+
\frac{r}{n}
\left(
aS_r^{n-1}+w-b
\right).
\label{eq:affine_record_payoff}
\end{align}
For \(r=0\), the rule accepts the first candidate, and hence $V_n(0)=\frac{a}{n}+b$.

The payoff increments are therefore
\begin{equation}\label{eq:affine_harmonic_increment}
\Delta_n(r)
:=
V_n(r)-V_n(r-1)
=
\frac{a}{n}
\left(
S_r^{n-1}-C
\right),
\end{equation}
where
\begin{equation}\label{eq:affine_harmonic_C}
C
:=\frac{a+b-w}{a}.
\end{equation}
We assume throughout that \(C>0\). The three models are obtained from this framework through the parameter
identifications in Table~\ref{tab:affine_record_models}. We follow the notation of~\cite{BayonEtAl2023}, except that we denote the
interview-cost parameter by \(\kappa\), with \(0\leq\kappa<1\).

\begin{table}[ht]
\centering
\caption{Parameter identifications for the fixed-horizon models with harmonic
threshold increments.}
\label{tab:affine_record_models}
\renewcommand{\arraystretch}{1.4}
\begin{tabular}{lcccc}
\toprule
Model & \(a\) & \(b\) & \(w\) & \(C\) \\
\midrule
Classical secretary
& \(1\)
& \(0\)
& \(0\)
& \(1\)
\\

Secretary problem with interview cost
& \(1-\kappa\)
& \(0\)
& \(-\kappa\)
& \(\dfrac{1}{1-\kappa}\)
\\

Win--Lose--or--Draw
& \(\alpha+\beta\)
& \(-\beta\)
& \(-\gamma\)
& \(\dfrac{\alpha+\gamma}{\alpha+\beta}\)
\\
\bottomrule
\end{tabular}
\end{table}

Since
\[
\Delta_n(r+1)-\Delta_n(r)
=
-\frac{a}{nr}<0,
\]
the payoff increments are strictly decreasing in \(r\). Hence they change sign
at most once, from positive to negative, and the threshold-payoff sequence
\(\{V_n(r)\}\) is unimodal.
Since \(
S_1^{n-1}\to \infty,
\) equation~\eqref{eq:affine_harmonic_increment} gives
\(
\Delta_n(1)>0
\)
for all sufficiently large \(n\). Hence
\(V_n(1)>V_n(0)\), so \(r_n^{\mathrm{opt}}\geq1\). Moreover,
\[
\Delta_n(n)=-\frac{aC}{n}<0,
\]
and therefore \(r_n^{\mathrm{opt}}\leq n-1\).
For such \(n\), the largest-maximizer convention and unimodality imply
\[
\Delta_n\bigl(r_n^{\mathrm{opt}}\bigr)\geq0,
\qquad
\Delta_n\bigl(r_n^{\mathrm{opt}}+1\bigr) < 0.
\]
Using \eqref{eq:affine_harmonic_increment}, these inequalities are
equivalent to
\begin{equation}\label{eq:harmonic_criterion}
S_{r_n^{\mathrm{opt}}}^{n-1}\geq C,
\qquad
S_{r_n^{\mathrm{opt}}+1}^{n-1}< C.
\end{equation}
In other words, \(r_n^{\mathrm{opt}}\) is the largest integer \(r\) for
which
\[
S_r^{n-1}\geq C.
\]

One must have $\lim_{n\to \infty} r_n^{\mathrm{opt}} = \infty$, as otherwise \(r_n^{\mathrm{opt}}\) would remain bounded along a subsequence, and the second inequality in \eqref{eq:harmonic_criterion} cannot hold. From these inequalities we have $C
\leq
S_{r_n^{\mathrm{opt}}}^{n-1}
\leq
C+\frac{1}{r_n^{\mathrm{opt}}}$ and hence
\[
S_{r_n^{\mathrm{opt}}}^{n-1}\longrightarrow C.
\]

Using \eqref{eq:EM_expansion}, we obtain the leading threshold proportion,
\[
 \lim_{n\to \infty }\frac{r_n^{\mathrm{opt}}}{n}= \ell =
\e^{-C}.
\]

Consider the rescaled increments
\[
D_n(r):=
\frac{n^2}{a}\Delta_n(r)=n\left(S_r^{n-1}-C\right).
\]
Let \(r=\ell n+u\),
where \(u\) remains in a fixed bounded interval and \(r\in I_n^*\) is an
integer threshold.  Applying \eqref{eq:EM_expansion}, uniformly for bounded
\(u\), using
\(-\log\ell=C\), gives
\begin{equation}\label{eq:fixed_harmonic_expansion}
D_n(\ell n+u)
=
-\frac1\ell
\left(
u-\frac{1-\ell}{2}
\right)
+
O(n^{-1}).
\end{equation}
Therefore, Theorem~\ref{thm:general_threshold_principle} applies with
\[
A=-\frac1\ell,
\qquad
c=\frac{1-\ell}{2}.
\]
\begin{remark}
For the classical secretary problem, \(C=1\), and hence
\[
\ell n+c
=
\frac{n}{\e}+\frac12-\frac1{2\e},
\]
which is precisely the upper endpoint in the classical
Gilbert--Mosteller bounds for the optimal threshold
\cite{GilbertMosteller1966}.
\end{remark}

% ------------------------------------------------------------
\subsection{The secretary problem with uncertain employment}
\label{subsec:uncertain_employment}
% ------------------------------------------------------------

We consider the uncertain-employment secretary problem introduced by Smith~\cite{Smith1975}. Each applicant independently accepts an offer with probability \(\rho\in(0,1)\), regardless of their rank and of the decisions of the other applicants. If an offer is declined, the search continues.

In~\cite{Smith1975} it is proved that an optimal rule is
of threshold type: reject an initial block and then offer the position to
each subsequent record until an offer is accepted. More precisely, if
\(s_{n,\rho}\) denotes the first stage at which an offer may be made, then
\(s_{n,\rho}\) is the smallest integer \(s\in\{1,\dots,n-1\}\) such that \cite[Eq.~(1)]{Smith1975},
\begin{equation}\label{eq:smith_cutoff}
\prod_{k=s}^{n-1}
\left(1+\frac{1-\rho}{k}\right)
\leq\frac1{\rho};
\end{equation}
Moreover, \cite[Eq.~(3)]{Smith1975},
\begin{equation}\label{eq:smith_first_order}
\frac{s_{n,\rho}}{n}
\longrightarrow
\ell_\rho
:=
\rho^{1/(1-\rho)}.
\end{equation}

In our convention, the threshold is the number of applicants initially
rejected. Hence Smith's rule has threshold \(s_{n,\rho}-1\). In the
exceptional case of equality in \eqref{eq:smith_cutoff}, the two adjacent
thresholds \(s_{n,\rho}-1\) and \(s_{n,\rho}\) may both be optimal, and our
largest-maximizer convention selects the latter. In either case, \(\left|r_{n,\rho}^{\mathrm{opt}}-s_{n,\rho}\right|\leq1
\), so \eqref{eq:smith_first_order} gives
\begin{equation}\label{eq:uncertain_employment_leading}
\frac{r_{n,\rho}^{\mathrm{opt}}}{n}
\longrightarrow
\ell_\rho.
\end{equation}

We now determine the constant-order correction. For a threshold \(r\), the
overall best applicant is selected at stage \(j>r\) precisely when she
accepts the offer and no earlier record after stage \(r\) has accepted.
Therefore,
\begin{equation}\label{eq:uncertain_employment_payoff_sum}
V_{n,\rho}(r)
=
\frac{\rho}{n}
\sum_{j=r+1}^{n}
\prod_{k=r+1}^{j-1}
\left(1-\frac{\rho}{k}\right).
\end{equation}
Writing the product in \eqref{eq:uncertain_employment_payoff_sum} in terms
of gamma functions and summing the resulting telescoping expression gives
\begin{equation}\label{eq:uncertain_employment_payoff}
V_{n,\rho}(r)
=
\frac{\rho}{n(1-\rho)}
\left[
\frac{\Gamma(r+1)\Gamma(n+1-\rho)}
     {\Gamma(r+1-\rho)\Gamma(n)}
-r
\right].
\end{equation}
Consequently, for \(1\leq r\leq n\),
\begin{equation}\label{eq:uncertain_employment_increment}
\Delta_{n,\rho}(r)
=
\frac{\rho}{n(1-\rho)}
\left[
\rho
\frac{\Gamma(r)\Gamma(n+1-\rho)}
     {\Gamma(r+1-\rho)\Gamma(n)}
-1
\right].
\end{equation}

Put
\[
R_{n,\rho}(r)
:=
\frac{\Gamma(r)\Gamma(n+1-\rho)}
     {\Gamma(r+1-\rho)\Gamma(n)}
=
\prod_{k=r}^{n-1}
\left(1+\frac{1-\rho}{k}\right),
\qquad 1\leq r\leq n.
\]
For \(1\leq r\leq n-1\),
\[
\frac{R_{n,\rho}(r+1)}{R_{n,\rho}(r)}
=
\frac{r}{r+1-\rho}
<1.
\]
Thus \(R_{n,\rho}(r)\), and hence
\(\Delta_{n,\rho}(r)\), is strictly decreasing in \(r\). Consequently, the
payoff sequence $\{V_{n,\rho}(r)\}_{r=0}^{n}$ is unimodal. Moreover, the sign-change criterion obtained from
\eqref{eq:uncertain_employment_increment} agrees with Smith's
criterion~\eqref{eq:smith_cutoff}.

Define the rescaled increments
\begin{equation}\label{eq:uncertain_employment_score}
D_{n,\rho}(r)
:=
\frac{n^2}{\rho}\Delta_{n,\rho}(r)
=
\frac{n}{1-\rho}
\left[
\rho R_{n,\rho}(r)-1
\right].
\end{equation}

By the standard asymptotic expansion for ratios of Gamma functions
\cite[\S5.11(iii)]{DLMF},
\[
\frac{\Gamma(z+\alpha)}{\Gamma(z+\beta)}
=
z^{\alpha-\beta}
\left[
1+\frac{(\alpha-\beta)(\alpha+\beta-1)}{2z}
+O(z^{-2})
\right],
\qquad z\to\infty,
\]
for fixed $\alpha,\beta$. Since $r=\ell_\rho n+u$ with $u$ bounded, we have $r\asymp n$ uniformly in $u$. Therefore
\begin{equation}\label{eq:uncertain_employment_gamma_expansion}
R_{n,\rho}(r)
=\frac{\Gamma(r)}{\Gamma(r+1-\rho)}\frac{\Gamma(n+1-\rho)}{\Gamma(n)}
=
\left(\frac rn\right)^{\rho-1}
\left[
1+\frac{\rho(1-\rho)}2
\left(\frac1r-\frac1n\right)
+O(n^{-2})
\right],
\end{equation}
uniformly for bounded $u$.
Since $\rho\ell_\rho^{\rho-1}=1$, equation~\eqref{eq:uncertain_employment_gamma_expansion} gives
\[
\rho R_{n,\rho}(\ell_\rho n+u)-1
=
\frac{1-\rho}{n}
\left[
-\frac{u}{\ell_\rho}
+
\frac{\rho}{2}
\left(\frac1{\ell_\rho}-1\right)
\right]
+
O(n^{-2}).
\]
Substituting this into \eqref{eq:uncertain_employment_score}, we obtain,
uniformly for bounded \(u\),
\begin{equation}\label{eq:uncertain_employment_local_expansion}
D_{n,\rho}(\ell_\rho n+u)
=
-\frac1{\ell_\rho}
\left(
u-c_\rho
\right)
+
O(n^{-1}),
\end{equation}
where
\begin{equation}\label{eq:uncertain_employment_shift}
c_\rho
:=
\frac{\rho(1-\ell_\rho)}{2}.
\end{equation}
Since \(0<\rho<1\) and \(0<\ell_\rho<1\), we have
$0<c_\rho<\frac12$.
Thus the hypotheses of
Theorem~\ref{thm:general_threshold_principle} are satisfied with
\[
\ell=\ell_\rho=\rho^{1/(1-\rho)},
\qquad
A=-\frac1{\ell_\rho},
\qquad
c=c_\rho=\frac{\rho(1-\ell_\rho)}{2}.
\]

% ------------------------------------------------------------
\subsection{Secretary problems with power-biased random horizons}
\label{subsec:power_biased}
% ------------------------------------------------------------

Throughout this subsection, \(a\geq0\) is fixed as \(n\to\infty\). We consider the following power-biased extension of the uniform
random-horizon secretary problem. 
\[
\mathbb P(N=k)
=
\frac{k^a}{\sum_{m=1}^{n}m^a},
\qquad
k=1,\dots,n.
\]
Although its first-order threshold equation
is covered by the general scaling-limit framework of Yasuda
\cite[Corollary~2.4]{Yasuda1984}, to the best of our knowledge this discrete
parametric family and its constant-order threshold correction have not
previously been analyzed explicitly.

The case \(a=0\) is the uniform random-horizon model.
For \(a=1\), the probability of a horizon of length \(k\) is proportional
to \(k\), corresponding to the size-biased version of the uniform horizon
distribution. More generally, \(a\) controls the degree to which the
distribution is biased toward larger candidate pools.

Conditional on \(N=k\), the success probability of threshold \(r\) is 
\[
P(k,r) :=
\mathbb P(X_{n,r}=1\mid N=k)
=
\begin{cases}
\dfrac1k,
& r=0,\\[2mm]
\dfrac{r}{k}\displaystyle\sum_{j=r}^{k-1}\frac1j,
& 1\leq r<k,\\[3mm]
0,
& k\leq r.
\end{cases}
\]
Indeed, when \(r=0\), the first applicant is selected, whereas for
\(1\leq r<k\), the overall best applicant must occur after time \(r\), and
the best among the applicants preceding it must occur within the initial
rejected block.

Since \(X_{n,r}\) is the indicator of successfully selecting the best
applicant,
\[
V_{n,a}(r)
=
\mathbb E[X_{n,r}]
=
\sum_{k=1}^{n}
\mathbb P(N=k)P(k,r).
\]
Thus
\[
V_{n,a}(r)
=
\begin{cases}
\dfrac{1}{W_{n,a}}
\displaystyle\sum_{k=1}^{n} k^{a-1},
& r=0, \\[3mm]
\dfrac{1}{W_{n,a}}
\displaystyle\sum_{k=r+1}^{n}
k^a\frac{r}{k}
\displaystyle\sum_{j=r}^{k-1}\frac1j,
& 1\leq r\leq n,
\end{cases}
\]
where
\[
W_{n,a}:=\sum_{m=1}^{n}m^a.
\]

Therefore, the discrete increments are
\begin{equation}\label{eq:power_delta}
\Delta_{n,a}(r)
=
V_{n,a}(r)-V_{n,a}(r-1)
=
\frac1{W_{n,a}}
\sum_{k=r}^{n}
k^{a-1}
\left(
\sum_{j=r}^{k-1}\frac1j-1
\right).
\end{equation}

We use the rescaled increments
\begin{equation}\label{eq:power_D_def}
D_{n,a}(r)
:= (W_{n,a}n^{1-a}) \Delta_{n,a}(r) =
n^{1-a}
\sum_{k=r}^{n}
k^{a-1}
\left(
\sum_{j=r}^{k-1}\frac1j-1
\right).
\end{equation}

Next, we verify that the assumptions in Theorem~\ref{thm:general_threshold_principle} are satisfied.

\begin{lemma}
\label{lem:power_biased_unimodality}
For every \(a\geq0\) and \(n\geq1\), the payoff sequence \(
\{V_{n,a}(r)\}_{r=0}^{n}
\)
is unimodal.
\end{lemma}

\begin{proof}
For \(1\leq r\leq n-1\),
\[
W_{n,a} \bigl( \Delta_{n,a}(r+1) - \Delta_{n,a}(r) \bigr) = \frac{1}{r} \left( r^a - \sum_{k=r+1}^{n} k^{a-1} \right).
\]
Since \(a \ge 0\), the term inside the parentheses is strictly increasing in $r$, so the
differences
\(
\Delta_{n,a}(r+1)-\Delta_{n,a}(r)
\)
change sign at most once, from negative to positive.  Consequently,
\(\Delta_{n,a}(r)\) first decreases and then increases. Since $\Delta_{n,a}(n) = -n^{a-1} / W_{n,a} < 0$, every term in the terminal increasing phase is negative. Therefore, $\Delta_{n,a}(r)$ crosses zero from nonnegative to nonpositive at most once, which shows that \(V_{n,a}(r)\) first increases and then decreases.
\end{proof}

We next study the asymptotic behavior of the rescaled increments
\(D_{n,a}\). The following two-term expansion identifies both the leading
threshold proportion and its constant-order correction.

\begin{lemma}
\label{lemma:power_D_expansion}
Fix \(a\geq0\). Uniformly when \(x=r/n\) ranges over a compact subinterval
of \((0,1)\),
\begin{equation}\label{eq:power_D_expansion_x}
D_{n,a}(r)
=
nI_a(r/n)+J_a(r/n)+O(n^{-1}),
\end{equation}
where
\begin{equation}\label{eq:power_I_definition}
I_a(x)
:=
\int_x^1
y^{a-1}
\left(
\log\frac{y}{x}-1
\right)
\,dy
\end{equation}
and
\begin{equation}\label{eq:power_J_definition}
J_a(x)
:=
\frac12
\left(
\log\frac1x-1-x^{a-1}
\right)
+
\frac{1}{2x}
\int_x^1 y^{a-1}\,dy
-
\frac12
\int_x^1 y^{a-2}\,dy.
\end{equation}
\end{lemma}

\begin{proof}
Using \(
\sum_{j=r}^{k-1}\frac1j
=
\log\frac{k}{r}
+
\frac1{2r}
-
\frac1{2k}
+
O(n^{-2})
\) uniformly for \(r\leq k\leq n\), with \(x=r/n\) bounded away from zero,
\begin{equation} \label{eq:expansion_EM}
\sum_{k=r}^{n}
k^{a-1}
\left(
\sum_{j=r}^{k-1}\frac1j-1
\right)=
\sum_{k=r}^{n}
k^{a-1}
\left(
\log\frac{k}{r}-1
\right)+
\frac1{2r}\sum_{k=r}^{n}k^{a-1}
-
\frac12\sum_{k=r}^{n}k^{a-2}
+
O(n^{a-2}).
\end{equation}
We expand each sum up to an error of order $O(n^{a-2})$ using Euler--Maclaurin summation,
\[
\sum_{k=r}^{n}
k^{a-1}
\left(
\log\frac{k}{r}-1
\right)
=
n^aI_a(x)
+
\frac{n^{a-1}}2
\left(
\log\frac1x-1-x^{a-1}
\right)
+
O(n^{a-2}),
\]
\[
\frac1{2r}\sum_{k=r}^{n}k^{a-1}
=
\frac{n^{a-1}}{2x}
\int_x^1y^{a-1}\,dy
+
O(n^{a-2}),
\]
and
\[
-\frac12\sum_{k=r}^{n}k^{a-2}
=
-\frac{n^{a-1}}2
\int_x^1y^{a-2}\,dy
+
O(n^{a-2}).
\]
Substituting these estimates in \eqref{eq:expansion_EM} and multiplying by \(n^{1-a}\) proves
\eqref{eq:power_D_expansion_x}.
\end{proof}

We next identify the constants \(A\), \(c\), and \(\ell\) required in
Theorem \ref{thm:general_threshold_principle}. We begin with the leading term
\(I_a\).

For \(a=0\),
\[
I_0(x)
=
\frac12\left(\log\frac1x\right)^2
-
\log\frac1x,
\]
whose nontrivial zero is \(
\ell_0=\e^{-2}.
\)
For \(a>0\), direct integration gives
\[
I_a(x)
=
\frac{(a+1)(x^a-1)-a\log x}{a^2},
\]
so \(I_a(x)=0\) is equivalent to
\begin{equation}\label{eq:ella_equation}
(a+1)(1-x^a)+a\log x=0.
\end{equation}
The left-hand side is strictly concave on \((0,1)\), tends to
\(-\infty\) as \(x\downarrow0\), vanishes at \(x=1\), and has derivative
\(-a^2<0\) at \(x=1\). Hence it has a unique zero \(\ell_a\in(0,1)\), with
\[
I_a(x)>0 \quad (0<x<\ell_a),
\qquad
I_a(x)<0 \quad (\ell_a<x<1).
\]

The same first-order threshold equation also follows from
\cite[Corollary~2.4]{Yasuda1984}.

Let \(r_{n,a}^{\mathrm{opt}}\) denote the optimal threshold for the power-biased model. For every sufficiently small \(\varepsilon>0\),
\eqref{eq:power_D_expansion_x} shows that the increments are eventually
positive at
\(
r_n^-:=\lfloor(\ell_a-\varepsilon)n\rfloor
\)
and negative at
\(
r_n^+:=\lceil(\ell_a+\varepsilon)n\rceil
\).
By Lemma~\ref{lem:power_biased_unimodality},
\(
r_n^-\leq r_{n,a}^{\mathrm{opt}}\leq r_n^+,
\)
and therefore
\begin{equation}\label{eq:power_leading_threshold}
\frac{r_{n,a}^{\mathrm{opt}}}{n}
\longrightarrow\ell_a.
\end{equation}

We now determine the constant-order correction. Define
\begin{equation}\label{eq:power_shift_definition}
c_a
:=
-\frac{J_a(\ell_a)}{I_a'(\ell_a)}.
\end{equation}
Proposition~\ref{prop:power_biased_shift} in Appendix~\ref{app:power_biased_shift} shows that 
\(
0<c_a<\frac12,
\)
for every \(a\geq0\).

Since \(I_a(\ell_a)=0\), \eqref{eq:power_D_expansion_x} and Taylor expansion
give, uniformly for bounded \(u\) such that \(\ell_an+u\) is an admissible
integer threshold,
\begin{equation}\label{eq:power_local_expansion}
D_{n,a}(\ell_an+u)
=
I_a'(\ell_a)(u-c_a)+O(n^{-1}).
\end{equation}
Moreover,
\[
I_0'(\ell_0)=-\e^2,
\qquad
I_a'(x)=\frac{(a+1)x^a-1}{ax}
\quad (a>0),
\]
and the preceding concavity argument implies
\begin{equation}\label{eq:Iprime_negative}
I_a'(\ell_a)<0
\qquad (a\geq0).
\end{equation}

In summary, the conditions of Theorem
\ref{thm:general_threshold_principle} are satisfied with
\[
\ell=\ell_a,
\qquad
A=I_a'(\ell_a),
\qquad
c=c_a.
\]

In particular,
$\ell_0=\e^{-2}$,
$c_0=\e^{-2}$,  recovering the known leading threshold proportion for the uniform
random-horizon classical secretary problem
\cite{PreSon72,RasmussenRobbins1975}.
For \(a=1\), the leading threshold proportion \(\ell_1\) is the \emph{rumour constant}
\(\vartheta\)~\cite{Sudbury1985}, characterized by
\begin{equation}\label{eq:rumour_constant}
    -\log\vartheta=2(1-\vartheta).
\end{equation}

\begin{remark}[Global optimality of the threshold rule]
The preceding analysis was formulated within the class of single-threshold
rules. For the power-biased model with \(a\geq0\), however, this restriction
entails no loss of optimality. Indeed, Petruccelli
\cite[Theorem~2.2]{Petruccelli1983} gives a sufficient condition for the
globally optimal rule, among all stopping rules based on relative ranks, to be
an \(s(r)\)-rule, that is, to reject an initial segment and then accept the
first subsequent record.

Writing
\[
p_k=\mathbb P(N=k)=\frac{k^a}{W_{n,a}},
\qquad
W_{n,a}=\sum_{m=1}^n m^a,
\]
Petruccelli's condition is satisfied whenever
\[
Q_k
:=
\frac{p_k}{\displaystyle\sum_{m=k+1}^n p_m/m},
\qquad 1\leq k\leq n-1,
\]
has the property that \(Q_k>1\) persists for all subsequent indices. In the
present model, \(
Q_k
=
k^a/\sum_{m=k+1}^n m^{a-1},
\)
and, for \(1\leq k\leq n-2\),
\[
\frac{Q_{k+1}}{Q_k}
=
\left(\frac{k+1}{k}\right)^a
\frac{\displaystyle\sum_{m=k+1}^n m^{a-1}}
     {\displaystyle\sum_{m=k+2}^n m^{a-1}}
>1,
\]
since \(a\geq0\). Thus \(Q_k\) is strictly increasing, and Petruccelli's
condition holds. Consequently, a globally optimal rank-based stopping rule is
a single record-threshold rule. 
\end{remark}

% ------------------------------------------------------------
\subsection{Uniform random-horizon win--lose--or--draw problems}
\label{subsec:uniform_wld}

We now consider the win--lose--or--draw problem when the number of applicants
\(N\) is uniformly distributed on \(\{1,\dots,n\}\). 
In the fixed-horizon version, a reward \(\alpha\) is obtained when the best applicant is selected, a penalty \(\beta\) when a non-best applicant is
selected, and a penalty \(\gamma\) when no applicant is selected; see
\cite[Sec.~2.6, Exercise~1]{Ferguson2008} for the symmetric model and
\cite[Sec.~3.6]{BayonEtAl2023} for the general \((\alpha,\beta,\gamma)\) formulation.

Let
\[
A_0:=\alpha+\beta,
\qquad
C:=\frac{\alpha+\gamma}{\alpha+\beta},
\]
and assume \(A_0>0\) and \(C>0\).

For a fixed horizon \(k\), let \(V^{(k)}(r)\) denote the expected payoff
of the threshold rule with threshold \(r\). Solving the recurrence in
\cite[Sec.~3.6]{BayonEtAl2023} (or using \eqref{eq:affine_record_payoff} with \(a=A_0\), \(b=-\beta\), \(w=-\gamma\), and with the fixed horizon \(n\) replaced by \(k\)) gives, for \(1\leq r\leq k\),
\[
V^{(k)}(r)
=
-\beta
+
\frac{r}{k}
\left[
A_0S_r^{k-1}+\beta-\gamma
\right].
\]
For \(r=0\), the rule accepts the first applicant, and hence
\[
V^{(k)}(0)
=
-\beta+\frac{A_0}{k}.
\]

For \(k<r\), both threshold rules \(r\) and \(r-1\) make no selection,
so their payoff difference is zero. For \(k\geq r\), a direct
calculation gives
\[
V^{(k)}(r)-V^{(k)}(r-1)
=
\frac{A_0}{k}
\left(S_r^{k-1}-C\right).
\]
Therefore, conditioning on \(N\), for \(1\leq r\leq n\),
\[
\Delta_n(r)
=
\frac{A_0}{n}
\sum_{k=r}^{n}\frac1k
\left(S_r^{k-1}-C\right)
\]
and thus
\[
\Delta_n(r+1)-\Delta_n(r)
=
\frac{A_0}{nr}
\left(
C-S_{r+1}^{n}
\right).
\]
Since \(S_{r+1}^{n}\) is strictly decreasing in \(r\), the right-hand side
changes sign at most once, from negative to positive. Hence
\(\Delta_n(r)\) first decreases and then increases. Moreover,
\[
\Delta_n(n)
=
-\frac{A_0C}{n^2}<0,
\]
so the terminal increasing part remains negative. Therefore
\(\Delta_n(r)\) changes sign at most once, from nonnegative to nonpositive,
and consequently
$\{V_n(r)\}_{r=0}^{n}$ is unimodal.

We define the rescaled increments
\begin{equation}\label{eq:uniform_wld_score}
D_n(r)
:=
\frac{n^2}{A_0}\Delta_n(r)
=
n\left[
\frac12\left((S_r^n)^2-T_r^n\right)
-CS_r^n
\right].
\end{equation}

Uniformly when \(x=r/n\) ranges over a compact subinterval of
\((0,1)\),
\[
\frac{D_n(r)}{n}
=
\frac12\left(\log\frac nr\right)^2
-
C\log\frac nr
+
O(n^{-1}).
\]
The limiting expression has zeros at \(x=1\) and
\[
x=\e^{-2C}.
\]
The former is a boundary zero, while the latter is the unique zero in
\((0,1)\).
Hence
\begin{equation}\label{eq:uniform_wld_constant}
\ell
=
\e^{-2C}
=
\exp\left(
-\frac{2(\alpha+\gamma)}{\alpha+\beta}
\right).
\end{equation}

The leading term in  \(D_n(r)\) is positive for \(r/n<\ell\) and negative for \(r/n>\ell\). By unimodality, the optimal threshold satisfies
\[
\frac{r_n^{\mathrm{opt}}}{n}
\longrightarrow
\ell= \e^{-2C}.
\]

We now determine the constant-order correction. Let $r=\ell n+u$, where \(u\) remains bounded. Using
\eqref{eq:EM_expansion}, \eqref{eq:T_EM_expansion},
we obtain, uniformly for bounded \(u\),
\begin{equation}\label{eq:uniform_wld_expansion}
D_n(\ell n+u)
=
-\frac{C}{\ell}(u-c)
+
O(n^{-1}),
\end{equation}
where
\begin{equation}\label{eq:uniform_wld_shift}
c=
\frac{C-1+(C+1)\ell}{2C}.
\end{equation}

It remains to check that \(0<c<1\). In fact, the stronger inequality $0<c<\frac12$
holds. Positivity is immediate when \(C\geq1\). If \(0<C<1\), then
\[
c>0
\quad\Longleftrightarrow\quad
2C<
\log\left(\frac{1+C}{1-C}\right)
=
2\int_0^C\frac{dt}{1-t^2},
\]
and the last integral is strictly greater than \(2C\). Moreover,
\[
\frac12-c
=
\frac{1-(C+1)\e^{-2C}}{2C}
>0.
\]
Thus the hypotheses of
Theorem~\ref{thm:general_threshold_principle} are satisfied with
\[
\ell=\e^{-2C},
\qquad
A=-\frac{C}{\ell}<0,
\qquad
c=\frac{C-1+(C+1)\ell}{2C}.
\]

\begin{remark}
The classical uniform random-horizon secretary problem corresponds to
\(\alpha=1\), \(\beta=0\), and \(\gamma=0\). Thus \(C=1\) and
\[
\ell=\e^{-2},
\qquad
c=\e^{-2}.
\]
The symmetric win--lose--or--draw problem corresponds to
\(\alpha=\beta=1\), \(\gamma=0\). In this case \(C=1/2\), and hence
\[
\ell=\e^{-1},
\qquad
c=\frac{3-\e}{2\e}.
\]
\end{remark}

% ------------------------------------------------------------
\subsection{Uniform random-horizon Best-or-Worst and Postdoc problems}
\label{subsec:uniform_bow_general}

We now consider the Best-or-Worst and Postdoc problems when the number of
applicants \(N\) is uniformly distributed on \(\{1,\dots,n\}\).
The fixed-horizon versions of these problems have been studied in
\cite{Rose1982ChoiceAssignment,Rose1982Nonextremal,Vanderbei2021}.
For the uniform random-horizon setting, Bay\'on et al.\
\cite{BayonEtAl2019} showed that the two problems have the same optimal
cutoff. We therefore carry out the analysis for the Best-or-Worst problem.

For a fixed horizon \(k\), the probability of success of the threshold rule
with threshold \(r\) is
\[
V^{(k)}(r)
=
\frac{2r(k-r)}{k(k-1)},
\qquad
1\leq r<k;
\]
see \cite[Prop.~2]{BayonEtAl2019}. Averaging over the uniform random horizon
gives, for \(1\leq r\leq n\),
\begin{equation}\label{eq:uniform_bow_payoff}
V_n(r)
=
\frac{2r}{n}
\left(
S_r^{n-1}+\frac{r}{n}-1
\right);
\end{equation}
see \cite[Sec.~4]{BayonEtAl2019}. Consequently, for \(2\leq r\leq n\),
\begin{equation}\label{eq:uniform_bow_increment}
\Delta_n(r)
=
V_n(r)-V_n(r-1)
=
\frac{2}{n}
\left(
S_r^{n-1}+\frac{2r-1}{n}-2
\right).
\end{equation}

We first verify unimodality on \(\{1,\dots,n\}\). For
\(2\leq r\leq n-1\),
\[
\Delta_n(r+1)-\Delta_n(r)
=
\frac{2}{n}
\left(
-\frac1r+\frac2n
\right).
\]
These differences change sign at most once, from negative to positive, so
\(\Delta_n(r)\) first decreases and then increases. Since
\[
\Delta_n(n)=-\frac{2}{n^2}<0,
\]
the terminal increasing part remains negative. Thus \(\Delta_n(r)\) changes
sign at most once, from nonnegative to nonpositive, and
$\{V_n(r)\}_{r=1}^{n}$ is unimodal.

We define
\begin{equation}\label{eq:uniform_bow_score}
D_n(r)
:=
\frac{n^2}{2}\Delta_n(r)
=
n\left(
S_r^{n-1}+\frac{2r-1}{n}-2
\right).
\end{equation}

Bay\'on et al.\ \cite[Thm.~5]{BayonEtAl2019} showed that the optimal
threshold satisfies
\[
\frac{r_n^{\mathrm{opt}}}{n}
\longrightarrow
\ell = \vartheta,
\]
where \(\vartheta\) is the rumour constant. 

The cutoff \(r=0\) is exceptional. Indeed,
\[
V_n(0)
=
\frac1n
\left(
1+2\sum_{k=2}^{n}\frac1k
\right)
\longrightarrow0.
\]
On the other hand, \eqref{eq:uniform_bow_payoff} and
\eqref{eq:rumour_constant} give
\[
V_n\bigl(\lfloor\vartheta n\rfloor\bigr)
\longrightarrow
2\vartheta(1-\vartheta)>0.
\]
Hence \(r=0\) is not optimal for all sufficiently large \(n\).

We now determine the constant-order correction. Let $r=\vartheta n+u$,
where \(u\) remains bounded. Using \eqref{eq:EM_expansion}, we obtain,
uniformly for bounded \(u\),
\begin{equation}\label{eq:uniform_bow_expansion}
D_n(\vartheta n+u)
=
\left(2-\frac1\vartheta\right)(u-c)
+
O(n^{-1}),
\end{equation}
where
\begin{equation}\label{eq:uniform_bow_shift}
c
=
\frac{1-3\vartheta}{2(1-2\vartheta)}.
\end{equation}
The defining equation \eqref{eq:rumour_constant} gives $\frac15<\vartheta<\frac14$, and therefore
\[
A:=2-\frac1\vartheta<0,
\qquad
0<c<\frac12<1.
\]

Thus all the hypotheses used in the proof of
Theorem~\ref{thm:general_threshold_principle} hold on the restricted
threshold set \(\{1,\dots,n\}\). Let
\[
q_j\longrightarrow\infty,
\qquad
p_j-\vartheta q_j\longrightarrow0.
\]
Since \(p_j\to\infty\), both \(p_j\) and \(p_j+1\) belong to
\(\{1,\dots,q_j\}\) for all sufficiently large \(j\). The local expansion
then gives
\[
D_{q_j}(p_j)>0,
\qquad
D_{q_j}(p_j+1)<0.
\]
By unimodality, \(p_j\) is the unique maximizer over
\(\{1,\dots,q_j\}\). Since the exceptional threshold \(r=0\) is not
optimal for all sufficiently large \(q_j\), it follows that \(p_j\) is
the unique optimal threshold over \(\{0,\dots,q_j\}\).

Finally, Bay\'on et al.~\cite[Cor.~1]{BayonEtAl2019} showed that, for
every threshold \(r>1\), the Best-or-Worst success probability is twice
the corresponding Postdoc success probability. The Postdoc success
probabilities at the exceptional thresholds \(r=0,1\) are
\(O((\log n)/n)\), whereas the success probability at a threshold
\(r=\vartheta n+O(1)\) tends to
\(\vartheta(1-\vartheta)>0\).
Consequently, the two problems have the same set of optimal thresholds for
all sufficiently large \(n\). In particular, for all sufficiently large
\(j\), the threshold \(p_j\) is also the unique optimal threshold for the
uniform random-horizon Postdoc problem with upper horizon \(q_j\).

\begin{remark}[Relation with a previous approximation]
Let \(M(n)\) denote the optimal cutoff in the notation of Bay\'on
et al.~\cite{BayonEtAl2019}. It is asserted in
\cite[Theorem~6(ii)]{BayonEtAl2019} that
\[
M(n)-n\vartheta
\longrightarrow
\frac{1}{4-2\e^{2-2\vartheta}}.
\]
This limit cannot hold literally: it would imply
\[
M(n+1)-M(n)\longrightarrow\vartheta,
\]
which is impossible because the left-hand side is integer-valued and
\(0<\vartheta<1\).

Their continuous calculation underlying the cited result is nevertheless
valid: if \(\alpha(n)\) denotes the real maximizer of the smooth
approximating payoff then 
\(
\alpha(n)-n\vartheta
\longrightarrow
\frac{1}{4-2\e^{2-2\vartheta}}= c-\frac12.
\)
Although this does not imply the corresponding limit for the integer
optimizer \(M(n)\), it identifies the correct continuous center. Indeed,
rounding \(
n\vartheta+c-\frac12
\) to the nearest integer gives, away from ties, \(
\lfloor n\vartheta+c\rfloor,
\) which agrees with the rounding law obtained here.
\end{remark}

% ============================================================
\subsection{A shifted case: the Poisson model}
\label{sec:poisson_general}

The preceding applications all fall within the rounding regime
\(0<c<1\) of Theorem~\ref{thm:general_threshold_principle}. We now
consider a model with the same leading threshold proportion as the
classical fixed-horizon secretary problem but with different
second-order behavior.

Suppose that the number of applicants \(N\) has a Poisson distribution
with mean \(\lambda>0\). This model was originally studied by Presman and
Sonin~\cite{PreSon72}. Conditional on \(N=k\), the classical fixed-horizon
success probability is
\[
V^{(k)}(r)
:=
\mathbb P(\text{success}\mid N=k)
=
\frac{r}{k}S_r^{k-1},
\qquad k>r\geq1,
\]
while \(V^{(k)}(r)=0\) for \(k\leq r\). For \(r=0\), the rule accepts the
first applicant, so its conditional success probability is \(1/k\) when
\(k\geq1\). Consequently,
\begin{equation}\label{eq:poisson_payoff}
P_\lambda(r)
=
\begin{cases}
\displaystyle
\e^{-\lambda}
\sum_{k=1}^{\infty}
\frac{\lambda^k}{k!\,k},
& r=0,\\[4mm]
\displaystyle
\e^{-\lambda}
\sum_{k=r+1}^{\infty}
\frac{\lambda^k}{k!}
\frac{r}{k}S_r^{k-1},
& r\geq1.
\end{cases}
\end{equation}

For \(r\geq1\), write
\[
\Delta_\lambda(r)
:=
P_\lambda(r)-P_\lambda(r-1).
\]

\begin{lemma}
\label{lem:poisson_unimodality}
For every \(\lambda>0\), the payoff sequence
\(\{P_\lambda(r)\}_{r\geq0}\) is unimodal.
\end{lemma}

\begin{proof}
Let
\[
\pi_k:=\mathbb P(N=k),
\qquad
H_\lambda(r)
:=
\pi_r-\sum_{k=r+1}^{\infty}\frac{\pi_k}{k}.
\]

A direct calculation gives
\[
\Delta_\lambda(r+1)-\Delta_\lambda(r)
=
\frac{H_\lambda(r)}{r},
\]
while
\[
H_\lambda(r+1)-H_\lambda(r)
=
\pi_r
\left(
\frac{\lambda(r+2)}{(r+1)^2}-1
\right).
\]
The expression in parentheses is strictly decreasing in \(r\), so
\(H_\lambda(r)\) first increases and then decreases. Moreover, for
\(r+1>\lambda\),
\[
0\leq
\frac1{\pi_r}\sum_{k=r+1}^{\infty}\frac{\pi_k}{k}
\leq
\frac{\lambda}
{(r+1)^2\left(1-\lambda/(r+1)\right)}
\longrightarrow0.
\]
Thus \(H_\lambda(r)>0\) eventually. Since its terminal decreasing part
is therefore positive, \(H_\lambda(r)\) changes sign at most once, from
negative to positive.
Consequently, \(\Delta_\lambda(r)\) first decreases
and then increases. Since
\[
0\leq P_\lambda(r)\leq\mathbb P(N>r)\longrightarrow0,
\]
we have \(\Delta_\lambda(r)\to0\). The terminal increasing part of
\(\Delta_\lambda\) is therefore nonpositive.
Hence
\(\Delta_\lambda(r)\) changes sign at most once, from nonnegative to
nonpositive, and \(\{P_\lambda(r)\}_{r\geq0}\) is unimodal.
\end{proof}

Let \(r_\lambda\) denote the optimal threshold. Presman and
Sonin~\cite{PreSon72} showed that
\[
r_\lambda
\sim
\frac{\lambda}{\e}.
\]
Since \(\lambda=\mathbb E[N]\) is the natural scale parameter of the
model, the analogue of the leading threshold proportion is therefore
\[
\ell
:=
\lim_{\lambda\to\infty}
\frac{r_\lambda}{\lambda}
=
\frac1{\e}.
\]

Define the rescaled payoff increments by
\[
D_\lambda(r)
:=
\lambda^2\Delta_\lambda(r),
\qquad r\geq1.
\]

\begin{lemma}[Poisson local expansion]
\label{lem:poisson_local_expansion}
Let \(\ell=1/\e\). Uniformly for bounded \(u\) such that
\(\ell\lambda+u\) is an integer threshold,
\begin{equation}\label{eq:poisson_local_expansion}
D_\lambda(\ell\lambda+u)
=
-\frac1\ell
\left(
u-c_{\mathrm P}
\right)
+
o(1)
\qquad
(\lambda\to\infty),
\end{equation}
where
\begin{equation}\label{eq:poisson_shift}
c_{\mathrm P}
=
\frac12-\frac2{\e}
\in(-1,0).
\end{equation}
\end{lemma}

\begin{proof}
For \(k\geq r\), subtracting the two fixed-horizon payoffs gives
\[
V^{(k)}(r)-V^{(k)}(r-1)
=
\frac1k\left(S_r^{k-1}-1\right).
\]
Consequently,
\[
\Delta_\lambda(r)
=
\mathbb E\left[
\frac1N\left(S_r^{N-1}-1\right);
\,N\geq r
\right].
\]

Let \(r=\ell\lambda+u\), with \(u\) bounded. Since \(\ell<1\), the
inequality
\[
|N-\lambda|\leq\lambda^{2/3}
\]
implies \(N\geq r\) for all sufficiently large \(\lambda\), uniformly
for bounded \(u\). Standard Poisson tail bounds, together with a bound
on the summand, show that restricting the expectation to this central
range produces an error \(o(\lambda^{-2})\).

On the central range, Euler--Maclaurin summation gives
\[
\frac1N\left(S_r^{N-1}-1\right)
=
f_r(N)+O(\lambda^{-3}),
\]
where
\[
f_r(x)
:=
\frac1x
\left(
\log\frac{x}{r}-1+\frac1{2r}-\frac1{2x}
\right).
\]
Taylor expansion about \(x=\lambda\), together with
\[
\mathbb E(N-\lambda)=0,
\qquad
\mathbb E(N-\lambda)^2=\lambda,
\]
and the exponential smallness of the complementary range, yields
\[
\Delta_\lambda(r)
=
f_r(\lambda)
+
\frac{\lambda}{2}f_r''(\lambda)
+
o(\lambda^{-2}),
\]
uniformly for bounded \(u\).
Direct expansion gives
\[
f_r(\lambda)
+
\frac{\lambda}{2}f_r''(\lambda)
=
\frac1{\lambda^2}
\left(
-\frac{u}{\ell}
+\frac1{2\ell}
-2
\right)
+
O(\lambda^{-3}),
\]
uniformly for bounded \(u\). Therefore,
\[
D_\lambda(\ell\lambda+u)
=
-\frac1\ell
\left(
u-(\frac12-\frac2{\e})
\right)
+
o(1)
\qquad
(\lambda\to\infty).
\]

\end{proof}

By Lemma~\ref{lem:poisson_unimodality}, the payoff sequence is unimodal,
and the sign-change argument used in
Remark~\ref{remark:Shifted} applies verbatim to the present infinite
threshold set, with shift \(-1\). Hence, if
\((p_j,q_j)\) are integer pairs such that
\[
q_j\longrightarrow\infty,
\qquad
p_j-\frac{1}{\e} q_j\longrightarrow0,
\]
then \(p_j-1\) is the unique optimal threshold for the Poisson model with
mean \(q_j\) for all sufficiently large \(j\).

% ============================================================

\appendix

\section{The shift in the power-biased model}
\label{app:power_biased_shift}

\begin{prop}
\label{prop:power_biased_shift}
For every \(a\geq0\),
\[
0<c_a<\frac12.
\]
\end{prop}

\begin{proof}
For \(a=0\), we have \(\ell_0=\e^{-2}\), and direct substitution in
\eqref{eq:power_J_definition} gives
\[
J_0(\ell_0)=1,
\qquad
I_0'(\ell_0)=-\e^2.
\]
Hence
\[
c_0=-\frac{J_0(\ell_0)}{I_0'(\ell_0)}
=\e^{-2}<\frac12.
\]
Suppose henceforth that \(a>0\). Write
\[
x:=\ell_a,
\qquad
h:=-1-\log x.
\]
The defining equation for \(\ell_a\) gives
\begin{equation}\label{eq:xh_relations}
x=\e^{-(1+h)},
\qquad
x^a=\frac{1-ah}{a+1},
\qquad
I_a'(x)=-\frac{h}{x}.
\end{equation}
Moreover, \(x<\e^{-1}\), since the left-hand side of
\eqref{eq:ella_equation} at \(x=\e^{-1}\) equals
\[
1-(a+1)\e^{-a}>0.
\]
Thus \(h>0\), while \eqref{eq:xh_relations} gives \(h<1/a\).

In fact,
\begin{equation}\label{eq:h_bounds}
\frac1{a+1}<h<\frac1a.
\end{equation}
Indeed, the function
\[
G_a(t):=\log\frac{a+1}{1-at}-a(1+t)
\]
is strictly increasing on \((0,1/a)\), satisfies \(G_a(h)=0\), and
\[
G_a\left(\frac1{a+1}\right)
=
2\log(a+1)-\frac{a(a+2)}{a+1}<0.
\]
Here the last inequality is the elementary bound
\(2\log t<t-t^{-1}\) for \(t>1\).

Implicit differentiation of \(G_a(h)=0\) gives
\[
h'(a)
=
-\frac{(1+h)((a+1)h-1)}{ah(a+1)}
<0.
\]
Let
\[
\psi(t):=1-t-2\e^{-(1+t)}.
\]
Since \(\psi'(t)<0\) for \(t>0\), and the equation defining \(\ell_1\)
gives \(\psi(h(1))=0\), we obtain
\begin{equation}\label{eq:psi_sign}
\operatorname{sgn}\psi(h(a))
=
\operatorname{sgn}(a-1).
\end{equation}

For \(a\neq1\), direct simplification using
\eqref{eq:xh_relations} yields
\begin{align}
2x(a^2-1)J_a(x)
&=
\psi(h)
+(a-1)\bigl(ah+x((a+1)h-1)\bigr),
\label{eq:J_sign_identity}\\
J_a(x)+\frac12I_a'(x)
&=
-\frac{1-(a-1)h-x^{a-1}}{2(a-1)}.
\label{eq:J_upper_identity}
\end{align}
By \eqref{eq:h_bounds}, the expression
\[
ah+x((a+1)h-1)
\]
is positive. Hence \eqref{eq:psi_sign} and
\eqref{eq:J_sign_identity} imply \(J_a(x)>0\).

To use \eqref{eq:J_upper_identity}, observe that if \(0<a<1\), with
\(d=1-a\),
\[
1-(a-1)h-x^{a-1}
=
1+dh-\e^{d(1+h)}<0.
\]
If \(a>1\), with \(b=a-1\), then \eqref{eq:h_bounds} gives
\[
\e^{-b(1+h)}
<
\frac1{1+b(1+h)}
<
1-bh,
\]
and consequently
\[
1-(a-1)h-x^{a-1}>0.
\]
Thus, in both cases,
\[
J_a(x)+\frac12I_a'(x)<0.
\]
Since \(J_a(x)>0\) and \(I_a'(x)<0\), it follows that
\[
0<c_a=-\frac{J_a(x)}{I_a'(x)}<\frac12
\qquad (a>0,\ a\neq1).
\]

Finally, when \(a=1\), direct substitution gives
\[
c_1=\frac{1-3\ell_1}{2(1-2\ell_1)}.
\]
The defining function
\[
-\log x-2(1-x)
\]
is negative at \(x=1/4\), so its interior zero satisfies
\(\ell_1<1/4\). The inequalities \(0<c_1<1/2\) follow immediately.
\end{proof}

% ============================================================
\section*{Acknowledgements}

The author thanks J.~M. Grau for useful conversations, for suggesting Corollary~\ref{cor:derangement_threshold}, and for proposing the
investigation of the continued-fraction phenomenon for the rumour constant.

% ============================================================
\section*{Declarations}

\textbf{Funding and Competing interests} The author has no relevant financial or non-financial interests to disclose, and no funds, grants, or other support were received during the preparation of this manuscript.

\noindent \textbf{Use of generative AI} During the preparation of this manuscript, the author used OpenAI's ChatGPT as an exploratory and editorial aid in developing and checking some of the examples in Section~\ref{sec:applications}. The general threshold principle stated in Theorem~\ref{thm:general_threshold_principle}, including its mathematical formulation and proof, is the author's original work. All AI-assisted material was independently checked and revised by the author, who takes full responsibility for the content of the manuscript.

\bibliographystyle{plain}
\bibliography{references}

% ============================================

\noindent{\small\sc 4i Intelligent Insights, Tecnoincubadora Marie Curie, PCT Cartuja, 41092 Sevilla, Spain}

\noindent E-mail: {\tt \href{mailto:r.sanchez@4i.ai}{r.sanchez@4i.ai}}

\end{document}